\documentclass[11pt,leqno]{amsart}
\usepackage{amssymb, amsmath}
\usepackage{pstricks,pst-node}

\let\oldsection=\section

\makeatletter
\renewcommand{\@seccntformat}[1]{\bf\@nameuse{the#1}.\quad}
\renewcommand\section{\@startsection{section}{1}%
            \z@{.7\linespacing\@plus\linespacing}{.5\linespacing}%
            {\normalfont\bfseries \boldmath}}
\renewcommand\subsection{\@startsection{subsection}{2}%
            \z@{.5\linespacing\@plus.7\linespacing}{-.5em}%
            {\normalfont\bfseries \boldmath}}
\renewcommand\subsubsection{\@startsection{subsubsection}{3}%
            \z@{.3\linespacing\@plus.5\linespacing}{-.5em}%
            {\normalfont\bfseries \boldmath}}
\makeatother

\theoremstyle{plain}
\newtheorem{theorem}{Theorem}

\newtheorem{lemma}[theorem]{Lemma}
\newtheorem{proposition}[theorem]{Proposition}
\newtheorem{conjecture}{Conjecture}

\theoremstyle{definition}

\theoremstyle{definition}

\numberwithin{equation}{subsection}

\newcounter{listequation}

\def\note#1{{\small\tt <<#1>>}}  % proof mode
\def\note#1{}              % final copy mode; suppresses notes

\def\i{{\mathbf i}}

\def\:{\colon}

\def\phi{\varphi}
\def\epsilon{\varepsilon}

\def\:{\colon}

\def\ft{{\mathfrak t}}

\def\cat{\operatorname{cat}}

\def\SO{\operatorname{SO}}
\def\Spin{\operatorname{Spin}}
\def\Sp{\operatorname{Sp}}

\def\O{\mathcal O}

\begin{document}
\title[The L-S category of simply connected compact Lie groups]{Distinguished orbits and the L-S category of\\ simply connected compact Lie groups}

\author{Markus Hunziker}
\address{Department of Mathematics \\
            Baylor University\\ Waco, Texas  }
%\address{\vskip-5.5ex}
\email{\tt \{Markus\underline{\ }Hunziker, Mark\underline{\ }Sepanski\}@baylor.edu}

\author{Mark R. Sepanski}
%\address{Department of Mathematics \\
%            Baylor University\\ Waco, Texas 76798 }
%\email{Mark\_Sepanski@baylor.edu}

\date{\today}



\dedicatory{}

\begin{abstract} 
We show that the Lusternik-Schnirelmann category of a simple, simply connected,
compact Lie group $G$ is bounded above by the sum of the relative categories of certain distinguished conjugacy classes in $G$ corresponding to the vertices
of the fundamental alcove for the action of the affine Weyl group on the Lie algebra of a maximal torus of $G$.

\end{abstract}

\maketitle

\parskip=2pt

%--------|---------|---------|---------|---------|---------|---------|---------
\section{Introduction}
%--------|---------|---------|---------|---------|---------|---------|---------

%\noindent
\subsection{}
The (normalized) {\it Lusternik-Schnirelmann category}  of a topological space $X$,
denoted $\cat(X)$,  is the least integer $m$  such that $X$ can be covered by $m+1$ open sets that are contractible in $X$. One of the problems on Ganea's
list (\cite{G}) from 1971  asks to find the L-S category of (compact) Lie groups. 
In 1975, Singhof (\cite{S1}) proved that 
$\cat(\operatorname{SU}(n+1))=n$. For the other families of simply connected compact 
Lie groups, the answer is only known when the rank is small
(cf. \cite{IMN} for a nice summary of what is known for simply connected and non-simply connected compact Lie groups of small rank.)

\subsection{}
%The purpose of this short note is to give an upper bound for the
%L-S category of $\operatorname{Sp}(n)$. 
%A general upper bound for the L-S category of a manifold is given by the dimension (cf. \cite{CGOT}.) Thus, $\cat(\operatorname{Sp}(n))\leq 2n^{2}+n$.
%Here we  prove that
%$$
%\cat(\operatorname{Sp}(n))\leq \left\lfloor \frac{(n+2)^2}{4}\right\rfloor-1,
%$$
%i.e., for $n=1,2,3,4,5,6,etc.$ our upper bound is 
%$1,3,5,8, 11,15,etc.$ For $n=1,2,3$ it is known (\cite{FGST}) that 
%$\cat(\operatorname{Sp}(n))=1,3,5$. 
%Also, for $n=1,2,3,4$ it is known (\cite{IK}) that 
%$\cat(\operatorname{Spin}(2n+1))=1,3,5,8$. 
%Based on this small set of data, we conjecture that the inequality
%above is in fact an equality.
%We remark that the best known lower bound is $\cat(\operatorname{Sp}(n))\geq n+2$ for $n\geq 3$ 
%(\cite{FGST},\cite{IM}).

%\subsection{}
The purpose of this short note is to show that the L-S category of a simple, simply connected, compact Lie group $G$ is bounded above by the sum of the relative categories of certain distinguished conjugacy classes in $G$. More precisely, 
suppose $\{v_{0}, \ldots, v_{n}\}$ are 
the vertices of the fundamental alcove for the action of the affine Weyl group on the Lie algebra of a maximal torus of $G$. For $0\leq k\leq n$, let $\O_{k}$ be the conjugacy class of $\exp v_{k}$ in $G$.
Then we will show in Section~\ref{S:cover} that
$$
   \cat(G) +1 \leq \sum_{k=0}^n\ (\cat_G (\O_k) +1),
$$
where $\cat_G (\O_k)$ is the {\it relative L-S category} of
$\O_{k}$ in $G$.
(If $Y\subseteq X$ is a topological subspace, $\cat_X(Y)$
is the least integer $m$ such that there there is a covering 
of $Y$ by $m + 1$ open subsets of $X$, each contractible in $X$.)

\subsection{}
For  $G=\operatorname{SU}(n+1)$, the conjugacy classes $\O_k$
turn out to be the points of the center of $G$ and we 
recover Singhof's result that
$\cat(\operatorname{SU}(n+1))\leq n$. 
For $G=\operatorname{Sp}(n)$, we
conjecture that 
$\cat_G (\O_k)\leq \min\{k,n-k\}$ (with respect to an appropriate numbering)  which would imply that
$$
\cat(\Sp(n))\leq \left\lfloor \frac{(n+2)^2}{4}\right\rfloor-1.
$$
Thus for $n=1,2,3,4,5,6,etc.$ our conjectured upper bound 
is  $1,3,5,8, 11,15,etc.$ For $n=1,2,3$ it is known (\cite{FGST}) that 
$\cat(\Sp(n))=1,3,5$. 
Also, for $n=1,2,3,4$ it is known (\cite{IK}) that 
$\cat(\Spin(2n+1))=1,3,5,8$. 
Based on this small set of data, we conjecture that 
$\cat(\Sp(n))=\cat(\Spin(2n+1))$ and that the inequality
above is in fact an equality.
We remark that the best known lower bound is $\cat(\Sp(n))\geq n+2$ for $n\geq 3$ 
(\cite{FGST},\cite{IM}).

\subsection{\it Acknowledgment}
The authors thank John Oprea who introduced us to the problem  during his visit to Baylor University in November 2008. The authors also thank the referee
for pointing out an error in an earlier version of the paper.

%--------|---------|---------|---------|---------|---------|---------|---------
\section{Notation}
%--------|---------|---------|---------|---------|---------|---------|---------

\subsection{}
Let $G$ be a simple, simply connected, compact Lie group with Lie algebra $%
\mathfrak{g}$. Let $T$ be a maximal torus of $G$ with Lie algebra $%
\mathfrak{t}$. Then $\mathfrak{h}=\mathfrak{t}_{\mathbb{C}}$ is a Cartan
subalgebra of $\mathfrak{g}_{\mathbb{C}}$ with $\mathfrak{h}_{\mathbb{R}}=\i%
\mathfrak{t}$.
Write $\Delta =\Delta (\mathfrak{g}_{\mathbb{C}},\mathfrak{h})$ for the set
of roots and choose a positive system $\Delta ^{+}$ with corresponding set of
simple roots $\Pi =\{\alpha _{1},\ldots ,\alpha _{n}\}$. With respect to
this system, write $\alpha _{0}$ for the highest root. For the classical
Lie groups and with respect to standard notation, $\Pi $ and $\alpha _{0}$ can be taken as in the following table:%
%\begin{table}
\begin{equation*}
\begin{tabular}{l|l|l}
$G$ & $\Pi$  & $\alpha _{0}$ \\ \hline
$\operatorname{SU}(n+1)$  & $%
\{\alpha _{i}=\epsilon _{i}-\epsilon _{i+1} \mid 1\leq i\leq n\}$ & $\epsilon
_{1}-\epsilon _{n+1}$ \\ 
$\operatorname{Sp}(n)$ &   $%
\{\alpha _{i}=\epsilon _{i}-\epsilon _{i+1} \mid 1\leq i\leq n-1\}  
\bigcup  \{\alpha _{n}=2\epsilon _{n}\}%
$ & $2\epsilon _{1}$ \\ 
$\operatorname{Spin}(2n+1)$  & $%
\{\alpha _{i}=\epsilon _{i}-\epsilon _{i+1} \mid 1\leq i\leq n-1\}  
\bigcup \,\{\alpha _{n}=\epsilon _{n}\}%
$ & $\epsilon _{1}+\epsilon _{2}$\\
$\operatorname{Spin}(2n)$  & $%
\{\alpha _{i}=\epsilon _{i}-\epsilon _{i+1} \mid 1\leq i\leq n-1\}  
\bigcup \,\{\alpha _{n}=\epsilon _{n-1}+\epsilon _{n}\}%
$ & $\epsilon _{1}+\epsilon _{2}$ \\ 
\end{tabular}%
\end{equation*}
%\end{table}

\smallskip

\subsection{}
Write $R^{\vee }$ for the coroot lattice in $\mathfrak{h}$ (which is the
same as the dual to the weight lattice in $\mathfrak{h}^{\ast }$) so that%
\begin{equation*}
R^{\vee }=\operatorname{span}_{\mathbb{Z}}\{h_{\alpha } \mid \alpha \in \Delta \}%
.
\end{equation*}%
Here $h_{\alpha }= 2u_{\alpha }/B(u_{\alpha },u_{\alpha })\in 
\mathfrak{h}_{\mathbb{R}}$ where $B(\cdot ,\cdot )$ is  the Killing
form and $u_{\alpha }\in \mathfrak{h}_{\mathbb{R}}$ is uniquely determined
by the equation $\alpha (H)=B(H,u_{\alpha })$ for all $H\in \mathfrak{h}_{%
\mathbb{R}}$. Since $G$ is simply connected, it follows that%
\begin{equation*}
\ker \left( \exp |_{\mathfrak{t}}\right) =2\pi\i R^{\vee }.
\end{equation*}

\subsection{}
The connected components of%
\begin{equation*}
\{t\in \mathfrak{t} \mid \alpha (t)\notin 2\pi\i \mathbb{Z}\text{ for }\alpha
\in \Delta \} 
\end{equation*}%
are called {\it alcoves}. Write $W=W(G,\mathfrak{t})$ for the Weyl group of $G$
with respect to $\mathfrak{t}$ viewed as acting on $\mathfrak{t}$ (and
extended to $\mathfrak{h}$ as needed). The {\it affine Weyl group}, $\widehat{W}$%
, is the group generated by the transformations of $\mathfrak{t}$ of the
form $t\mapsto wt+z$ for $w\in W$ and $z\in \ker \left( \exp |_{%
\mathfrak{t}}\right) $. It acts acts simply transitively on the set of
alcoves. The \emph{fundamental alcove}, $A_{0}$, is the alcove given by%
\begin{eqnarray*}
A_{0} &=&\{t=\i H\in \mathfrak{t} \mid 0<\alpha (H)<2\pi \text{ for }\alpha \in
\Delta ^{+}\} \\
&=&\{t=\i H\in \mathfrak{t} \mid \alpha _{0}(H)<2\pi \text{ and }0<\alpha
_{j}(H)\text{ for }1\leq j\leq n\}.
\end{eqnarray*}%
The closure of the fundamental alcove, $\overline{A}_{0}$,
is a fundamental domain for the $\widehat{W}$-action (cf. \cite[Thm.~4.8]{Hum}).
For $G=\Sp(2)$, the roots and the fundamental alcove are shown in Fig.~\ref{F:roots}.

\section{Cells}

\subsection{}
Define $v_{0}=0 \in \ft$ and for $1\leq k\leq n$, define $v_{k}\in \ft$ by the equations
$$
\alpha _{j}(v_{k})=
\begin{cases}
2\pi\i  & \text{if } j=0 \\ 
0 &  \text{if } 1\leq j\leq n \text{ and } j\neq k.
\end{cases}
$$
Then $\{v_{0},\ldots,v_{n}\}$ is the set of vertices of the $n$-simplex $\overline{A}_{0}$. 
Notice that if we write $\alpha _{0}=\sum_{j=1}^{n}m_{j}\alpha _{j}$ with $%
m_{j}\in \mathbb{N}$, we get $2\pi\i =\alpha
_{0}(v_{k})=\sum_{j=1}^{n}m_{j}\alpha _{j}(v_{k})=m_{k}\alpha _{k}(v_{k})$.
\ Therefore,
\begin{equation*}
\alpha _{k}(v_{k})=\frac{2\pi\i }{m_{k}}\quad \text{for } 1\leq k\leq n.
\end{equation*}%
(For classical $G$, the $m_{k}\in \{1,2\}$; however, for
exceptional $G$, the $m_{k}$ can be as large as $6$.)

\begin{figure}[ht]
\centering
\begin{pspicture}(-5,-2)(5,3)
$
%\psline[linewidth=.3pt](-5,0)(-4,0)
%\psline[linewidth=.3pt](-5,1)(-1,1)
%\psline[linewidth=.3pt](-5,2)(-1,2)
%\psline[linewidth=.3pt](-5,-1)(-1,-1)
%\psline[linewidth=.3pt](-5,-2)(-1,-2)
%\psline[linewidth=.3pt](-5,-2)(-5,2)
%\psline[linewidth=.3pt](-4,-2)(-4,2)
%\psline[linewidth=.3pt](-3,-2)(-3,2)
%\psline[linewidth=.3pt](-2,-2)(-2,2)
%\psline[linewidth=.3pt](-1,-2)(-1,2)
%\psline[linewidth=.3pt](-5,2)(-1,-2)
%\psline[linewidth=.3pt](-3,2)(-1,0)
%\psline[linewidth=.3pt](-5,0)(-3,-2)
%\psline[linewidth=.3pt](-5,-2)(-1,2)
%\psline[linewidth=.3pt](-3,-2)(-1,0)
%\psline[linewidth=.3pt](-5,0)(-3,2)
\psline[linewidth=1pt]{->}(-3,0)(-1,0)
\psline[linewidth=1pt]{->}(-3,0)(-2,1)
\psline[linewidth=1pt]{->}(-3,0)(-3,2)
\psline[linewidth=1pt]{->}(-3,0)(-4,1)
\psline[linewidth=1pt]{->}(-3,0)(-5,0)
\psline[linewidth=1pt]{->}(-3,0)(-4,-1)
\psline[linewidth=1pt]{->}(-3,0)(-3,-2)
\psline[linewidth=1pt]{->}(-3,0)(-2,-1)
\uput[r](-1,0){\alpha_{0}}
\uput[u](-3,2){\alpha_{2}}
\uput[d](-1.7,-.9){\alpha_{1}}
\pspolygon*[linecolor=lightgray](3,0)(4,0)(4,1)(3,0)
\psline[linewidth=.3pt](1,0)(5,0)
\psline[linewidth=.3pt](1,1)(5,1)
\psline[linewidth=.3pt](1,2)(5,2)
\psline[linewidth=.3pt](1,-1)(5,-1)
\psline[linewidth=.3pt](1,-2)(5,-2)
\psline[linewidth=.3pt](1,-2)(1,2)
\psline[linewidth=.3pt](2,-2)(2,2)
\psline[linewidth=.3pt](3,-2)(3,2)
\psline[linewidth=.3pt](4,-2)(4,2)
\psline[linewidth=.3pt](5,-2)(5,2)
\psline[linewidth=.3pt](1,2)(5,-2)
\psline[linewidth=.3pt](3,2)(5,0)
\psline[linewidth=.3pt](1,0)(3,-2)
\psline[linewidth=.3pt](1,-2)(5,2)
\psline[linewidth=.3pt](3,-2)(5,0)
\psline[linewidth=.3pt](1,0)(3,2)
%\psline[linewidth=.7pt](3,0)(4,1)
%\psline[linewidth=.7pt](4,1)(4,0)
%\psline[linewidth=.7pt](3,0)(4,0)
$
\end{pspicture}
\caption{Roots and alcoves for $\operatorname{Sp}(2)$}
\label{F:roots}
\end{figure}

\subsection{}
Define
\begin{equation*}
F_{0}=\{t=\i H\in \mathfrak{t} \mid \alpha _{0}(t)=2\pi\i \text{ and }0\leq
\alpha _{j}(H)\text{ for }1\leq j\leq n\} 
\end{equation*}%
and for $1\leq k\leq n$,
\begin{equation*}
F_{k}=\{t=\i H\in \mathfrak{t} \mid \alpha _{0}(H)\leq 2\pi \text{, }0\leq
\alpha _{j}(H)\text{ for }1\leq j\leq n\text{ with }j\neq k\text{, and }%
0=\alpha _{k}(t)\} 
\end{equation*}%
Then $\{F_{0},\ldots,F_{n}\}$ is the set of faces of $\overline{A}_{0}$. 
For $0\leq k \leq n$, we will call $F_{k}$ the \emph{face opposite to 
}$v_{k}$. In the following, we will write 
$r_{k}\in \widehat{W}$ for the reflection across $F_{k}$. 
Explicitly, $r_{0}(t)=t-(\alpha _{0}(t)-2\pi
i)h_{\alpha _{0}}$ and $r_{k}(t)=t-\alpha _{k}(t)h_{\alpha _{k}}$
for $1\leq k\leq n$.

\subsection{}
For $0\leq k\leq n$, let $\widehat{W}_{k}$ be the stabilizer of $v_{k}$,%
\begin{equation*}
\widehat{W}_{k}=\{w \in \widehat{W} \mid w (v_{k})=v_{k}\}. 
\end{equation*}

\begin{lemma}
\label{lem 1a} For  $0\leq k\leq n$, the group
$\widehat{W}_{k}$ is generated
by $\{r_{j} \mid 0\leq j\leq n$ and $j\neq k\}$ and%
\begin{equation*}
\{\text{alcoves $A$ such that $v_{k}\in \overline{A}$}\}=\{w (A_{0}) \mid w
\in \widehat{W}_{k}\}. 
\end{equation*}
\end{lemma}

\begin{proof}
For the first statement, recall that it is well known (cf. \cite[Ch.~4]{Hum}) that
the stabilizer of any point in $\overline{A}_{0}$ is generated by the set of
reflections across the alcove faces that contain the point. In particular, 
$v_{k}$ lies on every face except $F_{k}$ and the result follows. For the
second statement, observe that any alcove $A$ can be uniquely written as $%
A=w (A_{0})$ for some $w \in \widehat{W}$. Since the vertices of 
$w (A_{0})$ are $\{w (v_{j}) \mid 0\leq j\leq n\}$, it follows that 
$v_{k}\in \overline{A}$ if and only if $v_{k}=w (v_{j})$ for some $j$, $%
0\leq j\leq n$. Since $\overline{A}_{0}$ is a
fundamental domain for the action of $\widehat{W}$, 
$v_{k}=w (v_{j})$ if and only if $k=j$ if and only if $w \in 
\widehat{W}_{k}$ as desired.
\end{proof}

\subsection{}
For $0\leq k\leq n$, define
\begin{equation*}
C_{k}=\bigcup_{w \in \widehat{W}_{k}}w \left(\,\overline{A}_{0}\backslash
F_{k}\right). 
\end{equation*}
For $G=\Sp(2)$, the cells are shown in Fig~\ref{F:cells}.

By Lemma~\ref{lem 1a} and construction, the following result is immediate.

\begin{proposition}\label{lem 2b}  
\begin{itemize}
\item[(a)] $C_{k}$ is an open neighborhood of $v_{k}$ that
is contractible to $v_{k}$ via a straight line contraction. 
\item[(b)] Each alcove wall having nonempty intersection with $C_{k}$
contains $v_{k}$. 
\item[(c)] Suppose $u_{1},u_{2}\in C_{k}$ satisfy $u_{2}=w (u_{1})$
for some $w \in \widehat{W}$. Then $v_{k}=w (v_{k})$. 
\item[(d)] $\overline{A}_{0}\subseteq \bigcup_{k=0}^{n}C_{k}$.
\end{itemize}\qed
\end{proposition}

\begin{figure}[ht]
\centering
\begin{pspicture}(-2,-2)(2,2.5)
$
\pspolygon*[linecolor=lightgray](-6,-1)(-4,-1)(-4,1)(-6,1)(-6,-1)
\psline[linewidth=.3pt](-7,0)(-3,0)
\psline[linewidth=.3pt](-7,1)(-3,1)
\psline[linewidth=.3pt](-7,2)(-3,2)
\psline[linewidth=.3pt](-7,-1)(-3,-1)
\psline[linewidth=.3pt](-7,-2)(-3,-2)
\psline[linewidth=.3pt](-7,-2)(-7,2)
\psline[linewidth=.3pt](-6,-2)(-6,2)
\psline[linewidth=.3pt](-5,-2)(-5,2)
\psline[linewidth=.3pt](-4,-2)(-4,2)
\psline[linewidth=.3pt](-3,-2)(-3,2)
\psline[linewidth=.3pt](-7,2)(-3,-2)
\psline[linewidth=.3pt](-5,2)(-3,0)
\psline[linewidth=.3pt](-7,0)(-5,-2)
\psline[linewidth=.3pt](-7,-2)(-3,2)
\psline[linewidth=.3pt](-5,-2)(-3,0)
\psline[linewidth=.3pt](-7,0)(-5,2)
\psline[linewidth=.7pt](-5,0)(-4,1)
\psline[linewidth=.7pt](-4,1)(-4,0)
\psline[linewidth=.7pt](-5,0)(-4,0)
\pscircle*[linewidth=.5pt](-5,0){.07}
\pspolygon*[linecolor=lightgray](0,0)(1,-1)(2,0)(1,1)(0,0)
\psline[linewidth=.3pt](-2,0)(2,0)
\psline[linewidth=.3pt](-2,1)(2,1)
\psline[linewidth=.3pt](-2,2)(2,2)
\psline[linewidth=.3pt](-2,-1)(2,-1)
\psline[linewidth=.3pt](-2,-2)(2,-2)
\psline[linewidth=.3pt](-2,-2)(-2,2)
\psline[linewidth=.3pt](-1,-2)(-1,2)
\psline[linewidth=.3pt](0,-2)(0,2)
\psline[linewidth=.3pt](1,-2)(1,2)
\psline[linewidth=.3pt](2,-2)(2,2)
\psline[linewidth=.3pt](-2,2)(2,-2)
\psline[linewidth=.3pt](0,2)(2,0)
\psline[linewidth=.3pt](-2,0)(0,-2)
\psline[linewidth=.3pt](-2,-2)(2,2)
\psline[linewidth=.3pt](0,-2)(2,0)
\psline[linewidth=.3pt](-2,0)(0,2)
\psline[linewidth=.7pt](0,0)(1,1)
\psline[linewidth=.7pt](1,1)(1,0)
\psline[linewidth=.7pt](0,0)(1,0)
\pscircle*[linewidth=.5pt](1,0){.07}
\pspolygon*[linecolor=lightgray](5,0)(7,0)(7,2)(5,2)(5,0)
\psline[linewidth=.3pt](3,0)(7,0)
\psline[linewidth=.3pt](3,1)(7,1)
\psline[linewidth=.3pt](3,2)(7,2)
\psline[linewidth=.3pt](3,-1)(7,-1)
\psline[linewidth=.3pt](3,-2)(7,-2)
\psline[linewidth=.3pt](3,-2)(3,2)
\psline[linewidth=.3pt](4,-2)(4,2)
\psline[linewidth=.3pt](5,-2)(5,2)
\psline[linewidth=.3pt](6,-2)(6,2)
\psline[linewidth=.3pt](7,-2)(7,2)
\psline[linewidth=.3pt](3,2)(7,-2)
\psline[linewidth=.3pt](5,2)(7,0)
\psline[linewidth=.3pt](3,0)(5,-2)
\psline[linewidth=.3pt](3,-2)(7,2)
\psline[linewidth=.3pt](5,-2)(7,0)
\psline[linewidth=.3pt](3,0)(5,2)
\psline[linewidth=.7pt](5,0)(6,1)
\psline[linewidth=.7pt](6,1)(6,0)
\psline[linewidth=.7pt](5,0)(6,0)
\pscircle*[linewidth=.5pt](6,1){.07}
$
\end{pspicture}
\caption{The cells $C_0$, $C_1$, and $C_2$ for $\operatorname{Sp}(2)$}
\label{F:cells}
\end{figure}

\section{A Cover of $G$}\label{S:cover}

\subsection{}
For $0\leq k\leq n$, define
\begin{equation*}
U_{k}=\left\{ c_{g}(\exp t) \mid g\in G\text{, }t\in C_{k}\right\} 
\quad\text{and}
\quad
\mathcal{O}_{k}=\left\{ c_{g}(\exp v_{k}) \mid g\in G\right\}, 
\end{equation*}%
where $c_{g}(x)=gxg^{-1}$ for $g,x \in G$.

\begin{theorem}
\begin{itemize}
\item[(a)] $\{U_{k} \mid 0\leq k\leq n\}$ is an open cover of $G$. 
\item[(b)] $\mathcal{O}_{k}$ is a deformation retract of $U_{k}$.
\end{itemize}
\end{theorem}

\begin{proof}
Since $\exp (C_{k})$ is open in $T$ and since conjugation takes the
exponential of the closure of an alcove onto $G$, part (a) is automatic. \
For part (b), we claim the deformation retract is given by $%
R_{k}:U_{k}\times I\rightarrow U_{k}$ where $I=[0,1]$ and%
\begin{equation*}
R_{k}(c_{g}(\exp t),s)=c_{g}\left( \exp \left( (1-s)t+sv_{k}\right) \right) \text{%
.} 
\end{equation*}%
It remains to see that $R_{k}$ is actually well defined.

Suppose $c_{g_{1}}(\exp t_{1})=c_{g_{2}}(\exp t_{2})$ for $g_{j}\in G$ and $%
t_{j}\in C_{k}$. Writing $c_{g_{2}^{-1}g_{1}}(\exp t_{1})=\exp t_{2}$,
there exists $h\in Z_{G}(\exp t_{2})^{0}$ so that $\widetilde{w}%
=hg_{2}^{-1}g_{1}\in N_{G}(T)$ (cf. \cite[Section 6.4]{Sep}.) Let $\Sigma
_{t_{2}}=\{\alpha \in \Delta  \mid \alpha (t_{2})\in 2\pi\i \mathbb{Z}\}$,
i.e., the set of $\alpha $ for which $t_{2}$ lies on an $\alpha $-alcove
wall. Then $Z_{G}(\exp t_{2})^{0}$ is the exponential of the direct sum of 
$\mathfrak{t}$ and all $\mathfrak{su}(2)$-triples corresponding to roots in $%
\Sigma_{t_{2}}$. Since $v_{k}$ also lies on all such $\alpha $-alcove
walls, it follows that $h\in Z_{G}(\exp \left( (1-s)t+sv_{k}\right) )^{0}$.

Setting $w=\operatorname{Ad}_{\widetilde{w}}\in W$, we have $c_{\widetilde{w}%
}(\exp t_{1})=\exp t_{2}$. Thus $\exp (wt_{1})=\exp (t_{2})$ so that $%
t_{2}=wt_{1}+z$ for some $z\in \ker \left( \exp |_{\mathfrak{t}}\right) $. 
By Proposition~\ref{lem 2b}, it follows that $v_{k}=wv_{k}+z$. Then%
\begin{eqnarray*}
c_{g_{1}}\left( \exp \left( (1-s)t_{1}+sv_{k}\right) \right)
&=&c_{g_{2}h^{-1}\widetilde{w}}\left( \exp \left( (1-s)t_{1}+sv_{k}\right)
\right) \\
&=&c_{g_{2}h^{-1}}\left( \exp \left( (1-s)wt_{1}+swv_{k}\right) \right) \\
&=&c_{g_{2}h^{-1}}\left( \exp \left( (1-s)\left( t_{2}-z\right) +s\left(
v_{k}-z\right) \right) \right) \\
&=&c_{g_{2}h^{-1}}\left( \exp \left( (1-s)t_{2}+sv_{k}-z\right) \right) \\
&=&c_{g_{2}}\left( \exp \left( (1-s)t_{2}+sv_{k}\right) \right)
\end{eqnarray*}%
and we are finished.
\end{proof}

\subsection{}
The results of the previous subsection give immediately the following main result.
\begin{theorem}
\label{thm: main}%
\begin{equation*}
\operatorname{cat}(G)+1\leq \sum_{k=0}^{n}\left( \operatorname{cat}_{G}\left( 
\mathcal{O}_{k}\right) +1\right) . 
\end{equation*}\qed
\end{theorem}

\section{The Orbits $\mathcal{O}_{k}$}

\noindent
We present some remarks and explicit realizations for the $\mathcal{O}_{k}$
in the classical cases.

\subsection{$G=\operatorname{SU}(n+1)$}

Trivial calculations show that%
\begin{equation*}
v_{k}=\frac{2\pi\i }{n+1}(\overset{k}{\overbrace{n+1-k,\ldots ,n+1-k}}%
,\,-k,\ldots ,-k) 
\end{equation*}%
for $0\leq k\leq n$. Therefore $\exp v_{k}=e^{\frac{-2\pi\i k}{n+1}}%
\operatorname{Id}$. In particular, $\mathcal{O}_{k}=\{e^{\frac{-2\pi\i k}{n+1}}%
\operatorname{Id}\}$ and so $\operatorname{cat}\left( \mathcal{O}_{k}\right) =0$ for all $0\leq k\leq n$. Thus, Theorem~\ref{thm: main} implies 
$\operatorname{cat}(\operatorname{SU}(n+1))\leq n$,
i.e., we recover Singhof's  result \cite{S1}.

\subsection{$G=\operatorname{Sp}(n)$}

%\subsubsection{The Orbits $\mathcal{O}_{k}$ for $\operatorname{Sp}(n)$}

Let $\mathbb{H}$ denote the division algebra of quaternions $q=a+b{\mathbf{i}%
}+c\mathbf{j}+d\mathbf{k}$, $a,b,c,d\in \mathbb{R}$. View $\mathbb{H}^{n}$
as a right vector space and identify the set of quaternionic matrices, $%
M_{n}(\mathbb{H})$, with the set of $\mathbb{H}$-linear endomorphisms of $%
\mathbb{H}^{n}$ via standard matrix multiplication on the left. Write $\nu
:M_{n}(\mathbb{H})\rightarrow \mathbb{R}$ for the reduced norm. In
particular, if $\varphi :M_{n}(\mathbb{H})\rightarrow M_{2n}(\mathbb{C})$ is
the $\mathbb{C}$-linear injective homomorphism given by 
\begin{equation*}
\varphi (A+\mathbf{j}B)=\left( 
\begin{array}{cc}
A & B \\ 
-\bar{B} & \bar{A}%
\end{array}%
\right) 
\end{equation*}%
for $A,B\in M_{n}(\mathbb{C})$, then $\nu =\det \circ \varphi $. We then
realize $\operatorname{GL}(n,\mathbb{H})=\{g\in M_{n}(\mathbb{H}) \mid \nu (g)\not=0\}$, $%
\operatorname{SL}(n,\mathbb{H})=\{g\in M_{n}(\mathbb{H}) \mid \nu (g)=1\}$, and 
\begin{equation*}
G=\operatorname{Sp}(n)=\{g\in \operatorname{SL}(n,\mathbb{H}) \mid gg^{\ast }=I_{n}\}\text{,}
\end{equation*}%
where $g^{\ast }$ denotes the quaternionic conjugate transpose of $g$. We
also fix the maximal torus 
\begin{equation*}
T=\{\operatorname{diag}(e^{\mathbf{i}\theta _{1}},\ldots ,e^{\mathbf{i}\theta
_{n}}) \mid \theta _{j}\in \mathbb{R}\}.
\end{equation*}

With this set-up, it is straightforward to check that%
\begin{equation*}
v_{k}=\i\pi \operatorname{diag}(\overset{k}{\overbrace{1,\ldots ,1}},0,\ldots ,0)
\end{equation*}%
for $0\leq k\leq n$.  Therefore 
\begin{equation*}
\exp v_{k}=\left( 
\begin{array}{cc}
-I_{k} &  \\ 
& I_{n-k}%
\end{array}%
\right) .
\end{equation*}%
In particular, $\mathcal{O}_{0}=\{\operatorname{Id}\}$ and $\mathcal{O}_{n}=\{-%
\operatorname{Id}\}$ so that $\operatorname{cat}(\mathcal{O}_{0})=\operatorname{cat}(\mathcal{O}_{n})=0$.

The other $\mathcal{O}_{k}$ require more work, though they are easy to
identify. For this we realize the quaternionic Grassmannian of $k$-planes
in $\mathbb{H}^{n}$, ${Gr}_{k}(\mathbb{H}^{n})$, by $%
\left\{ x\in {M}_{n\times k}(\mathbb{H}) \mid \operatorname{rk}%
(x)=k\right\} $ equipped with the equivalence relation $x\sim xh$ where $%
x\in {M}_{n\times k}(\mathbb{H}^{n})$ and $h\in \operatorname{GL}(k,%
\mathbb{H})$. The following result is immediate.

\begin{lemma}
Let $1\leq k\leq n-1$ and set $d_{k}=\min \{k,n-k\}$. Then there is a diffeomorphism $\tau _{k}:\mathcal{O}%
_{k}\rightarrow {Gr}_{d_{k}}(\mathbb{H}^{n})$,%
\begin{equation*}
\mathcal{O}_{k}\cong \operatorname{Sp}(n)/\left( \operatorname{Sp}(k)\times \operatorname{Sp%
}(n-k)\right) \cong {Gr}_{d_{k}}(\mathbb{H}^{n})\text{,} 
\end{equation*}%
given by%
\begin{equation*}
\tau _{k}\left( c_{g}(\exp v_{k})\right) =g\left( 
\begin{array}{c}
I_{k} \\ 
0_{(n-k)\times k}%
\end{array}%
\right) 
\end{equation*}%
when $d_{k}=k$ and by%
\begin{equation*}
\tau _{k}\left( c_{g}(\exp v_{k})\right) =g\left( 
\begin{array}{c}
0_{k\times (n-k)} \\ 
I_{n-k}%
\end{array}%
\right) 
\end{equation*}%
when $d_{k}=n-k$.\qed
\end{lemma}

\begin{conjecture}
$\operatorname{cat}_{\operatorname{Sp}(n)}(\mathcal{O}_{k})=d_{k}$.
\end{conjecture}

As we observed already in the introduction, if the conjecture is true, then Theorem~\ref{thm: main} quickly shows that 
\begin{equation*}
\operatorname{cat}\left( \operatorname{Sp}(n)\right) \leq \left\lfloor \frac{(n+2)^{2}}{4}%
\right\rfloor -1.
\end{equation*}%

%%\subsubsection{An Open Cover of $\mathcal{O}_{k}$}

In terms of trying to show that $\operatorname{cat}_{\operatorname{Sp}(n)}(\mathcal{O})%
_{k}\leq d_{k}$, there is an obvious choice of a cover of $\mathcal{O}_{k}$.
\ For this, we introduce the following notation. For the sake of clarity,
we assume we are in the case of $d_{k}=k$, i.e., $1\leq k\leq n/2$.

For $1\leq j\leq k+1$, write $x\in {Gr}_{k-1}(\mathbb{H}^{n-1})$ as%
\begin{equation*}
x=\left( 
\begin{array}{c}
x_{j,1} \\ 
x_{j,2}%
\end{array}%
\right) 
\end{equation*}%
with $x_{j,1}\in M_{(j-1)\times (k-1)}(\mathbb{H})$ and $x_{j,2}\in
M_{(n-j)\times (k-1)}(\mathbb{H})$. Let $X_{j,k}\cong {Gr}_{k-1}(\mathbb{H}%
^{n-1})\subseteq {Gr}_{k}(\mathbb{H}^{n})$ be given by%
\begin{equation*}
\{\left( 
\begin{array}{cc}
0_{(j-1)\times 1} & x_{j,1} \\ 
1 & 0_{1\times (k-1)} \\ 
0_{(k-j)\times 1} & x_{j,2}%
\end{array}%
\right)  \mid x\in {Gr}_{k-1}(\mathbb{H}^{n-1})\}.
\end{equation*}%
Write $y\in {Gr}_{k}(\mathbb{H}^{n-1})$ as%
\begin{equation*}
y=\left( 
\begin{array}{c}
y_{j,1} \\ 
y_{j,2}%
\end{array}%
\right) 
\end{equation*}%
with $y_{j,1}\in M_{(j-1)\times k}(\mathbb{H})$ and $y_{j,2}\in M_{(n-j)\times k}(\mathbb{H})$. Let $Y_{j,k}\cong {Gr}_{k}(\mathbb{H}^{n-1})\subseteq {%
Gr}_{k}(\mathbb{H}^{n})$ be given by%
\begin{equation*}
\{\left( 
\begin{array}{c}
y_{j,1} \\ 
0_{1\times k} \\ 
y_{j,2}%
\end{array}%
\right)  \mid y\in {Gr}_{k}(\mathbb{H}^{n-1})\}.
\end{equation*}

\begin{proposition}\label{P:Grassmannian}
\begin{itemize}
\item[(a)] $\{{Gr}_{k}(\mathbb{H}^{n})\backslash
X_{j,k} \mid 1\leq j\leq k+1\}$ is an open cover of ${Gr}%
_{k}(\mathbb{H}^{n})$. 
\item[(b)] $Y_{j,k}$ is a deformation retract of ${Gr}%
_{k}(\mathbb{H}^{n})\backslash X_{j,k}$. 
\item[(c)] Written in $(j-1)\times 1\times (n-j)$ block form, $\tau
_{k}^{-1}(Y_{j,k})$ is%
\begin{equation*}
\{\left( 
\begin{array}{ccc}
A &  & B \\ 
& 1 &  \\ 
C &  & D%
\end{array}%
\right)  \mid \left( 
\begin{array}{cc}
A & B \\ 
C & D%
\end{array}%
\right) \in \operatorname{Sp}(n-1)\text{ and conjugate to }\exp v_{k-1,n-1}\} 
\end{equation*}%
where $v_{k,n}=
\i\operatorname{diag}(\overset{k}{\overbrace{\pi ,\ldots ,\pi }},
\overset{n-k}{\overbrace{0,\ldots ,0})}$.
\end{itemize}
\end{proposition}

\begin{proof}
For part (a), simply observe that a $k$-plane in $X_{1,k}\cap \cdots \cap
X_{k+1,k}$ would have to contain $k+1$ independent vectors which is
impossible. For part (b), observe that ${Gr}_{k}(\mathbb{H%
}^{n})\backslash X_{j,k}$ is the of the set of%
\begin{equation*}
\left( 
\begin{array}{c}
x_{(j-1)\times k} \\ 
y_{1\times k} \\ 
z_{(n-j)\times k}%
\end{array}%
\right) \in {Gr}_{k}(\mathbb{H}^{n})\text{ so that }\left( 
\begin{array}{c}
x_{(j-1)\times k} \\ 
z_{(n-j)\times k}%
\end{array}%
\right) \in {Gr}_{k}(\mathbb{H}^{n-1}). 
\end{equation*}%
Therefore, the retraction $R:{Gr}_{k}(\mathbb{H}%
^{n})\backslash X_{j,k}\times I\rightarrow X_{j,k}$ given by%
\begin{equation*}
R(\left( 
\begin{array}{c}
x_{(j-1)\times k} \\ 
y_{1\times k} \\ 
z_{(n-j)\times k}%
\end{array}%
\right) ,s)=\left( 
\begin{array}{c}
x_{(j-1)\times k} \\ 
(1-s)y_{1\times k} \\ 
z_{(n-j)\times k}%
\end{array}%
\right) 
\end{equation*}%
does the trick. For part (c), observe that $\tau _{k}^{-1}(Y_{j,k})$ can
be written in $(j-1)\times 1\times (n-k)$ block form as%
\begin{equation*}
\{g=\left( 
\begin{array}{ccc}
\alpha & \beta & \gamma \\ 
0 & \delta & \zeta \\ 
\eta & \iota & \kappa%
\end{array}%
\right) \in G\}. 
\end{equation*}%
Making note that $gg^{\ast }=I$, part (c) follows immediately by explicit
matrix multiplication using $(j-1)\times 1\times (k-j)\times (n-j)$ block
form when $j\leq k$ and by using $k\times 1\times (n-k-1)$ block form when $%
j=k+1$.
\end{proof}

\begin{proposition}
If the sets $\tau _{k}^{-1}(Y_{j,k})$
%\begin{equation*}
%\{\left( 
%\begin{array}{ccc}
%A &  & B \\ 
%& 1 &  \\ 
%C &  & D%
%\end{array}%
%\right)  \mid \left( 
%\begin{array}{cc}
%A & B \\ 
%C & D%
%\end{array}%
%\right) \in \operatorname{Sp}(n-1)\text{ and conjugate to }\exp v_{k-1,n-1}\} 
%\end{equation*}%
are contractible in $\operatorname{SL}(n,\mathbb{H})$, then $\operatorname{cat}_{\operatorname{Sp}(n)}(\mathcal{O}_{k})\leq k$.
\end{proposition}

\begin{proof}
Let $F_{1}: \tau _{k}^{-1}(Y_{j,k})\times I\rightarrow \operatorname{SL}(n,\mathbb{H})$ be a contraction that takes 
$\tau _{k}^{-1}(Y_{j,k})$ to a point.
Using the Cartan decomposition, there is a diffeomorphism 
$\operatorname{SL}(n,\mathbb{H})\cong G\times \mathfrak{p}$ where 
$\mathfrak{p}$ is the the $-1$ eigenspace
of the Cartan involution corresponding to $\mathfrak{sp}(n)$, i.e., 
the involution given by $\theta \left(
x\right) =-x^{\ast }$. For $g\in \operatorname{SL}(n,\mathbb{H})$, uniquely write $%
g=\kappa (g)\exp (\rho (g))$ with $\kappa (g)\in G$ and $\rho (g)\in 
\mathfrak{p}$. Finally, define $F_{2}:\tau _{k}^{-1}(Y_{j,k})\times
I\rightarrow G$ by $F_{2}(g,s)=\kappa \left( F_{1}(g,s)\right) $. By
construction, $F_{2}$ contracts $\tau _{k}^{-1}(Y_{j,k})$ to a point.
Thus, if  the sets $\tau _{k}^{-1}(Y_{j,k})$ are contractible in 
$\operatorname{SL}(n,\mathbb{H})$ then they are also contractible in
$G=\operatorname{Sp}(n)$. The proposition then follows from Proposition~\ref{P:Grassmannian}.
 \end{proof}

At the present time, we do not know whether $\tau _{k}^{-1}(Y_{j,k})$ is
contractible in $\operatorname{SL}(n,\mathbb{H})$. It is worth noting that a similar
result can be obtained by showing that $\tau _{k}^{-1}(Y_{j,k})$ is
contractible in $\operatorname{Sp}(2n,\mathbb{C})$. This too is unknown.

\subsection{$G=\operatorname{Spin}(2n+1)$}

Write the tensor algebra over $\mathbb{R}^{m}$ as $\mathcal{T}_{m}(\mathbb{R}%
)$. Then the Clifford algebra is $\mathcal{C}_{m}(\mathbb{R})=\mathcal{T}%
_{m}(\mathbb{R})/\mathcal{I}$ where $\mathcal{I}$ is the ideal of $\mathcal{T%
}_{m}(\mathbb{R})$ generated by $\{(x\otimes x+\left\Vert x\right\Vert
^{2}) \mid x\in \mathbb{R}^{m}\}$. By way of notation for Clifford
multiplication, write $x_{1}x_{2}\cdots x_{k}$ for the element $x_{1}\otimes
x_{2}\otimes \cdots \otimes x_{k}+\mathcal{I}\in \mathcal{C}_{m}(\mathbb{R})$
where $x_{1},x_{2},\ldots ,x_{m}\in \mathbb{R}^{m}$. Write $\mathcal{C}%
_{m}^{+}(\mathbb{R})$ for the subalgebra of $\mathcal{C}_{m}(\mathbb{R})$
spanned by all products of an even number of elements of $\mathbb{R}^{m}$. \
Conjugation, an anti-involution on $\mathcal{C}_{m}(\mathbb{R})$, is defined
by $(x_{1}x_{2}\cdots x_{k})^{\ast }=(-1)^{k}\,x_{k}\cdots x_{2}x_{1}$ for $%
x_{i}\in \mathbb{R}^{m}$.

Then%
\begin{equation*}
\operatorname{Spin}(m)=\{g\in \mathcal{C}_{m}^{+}(\mathbb{R})\mid
gg^{\ast }=1\text{ and }gxg^{\ast }\in \mathbb{R}^{m}\text{ for all }x\in 
\mathbb{R}^{m}\}.
\end{equation*}%
In fact, it is the case that $\operatorname{Spin}(m)=\{x_{1}x_{2}\cdots
x_{2k} \mid x_{i}\in S^{m-1}$ for $2\leq 2k\leq 2m\}$. If we write $\left( 
\mathcal{A}g\right) x=gxg^{\ast }$ when $g\in \operatorname{Spin}(m)$
and $x\in \mathbb{R}^{m}$, then $\mathcal{A}$ gives the double cover of $%
\operatorname{SO}(m)$:%
\begin{equation*}
\{1\}\rightarrow \{\pm 1\}\rightarrow \operatorname{Spin}(m)\overset{%
\mathcal{A}}{\rightarrow }\operatorname{SO}(m)\rightarrow \{I_{m}\}.
\end{equation*}

A maximal torus\ $T_{0}$ for $\operatorname{SO}(2n+1)$ is given by%
\begin{equation*}
T_{0}=\{\left( 
\begin{array}{cccccc}
\cos \theta _{1} & \sin \theta _{1} &  &  &  &  \\ 
-\sin \theta _{1} & \cos \theta _{1} &  &  &  &  \\ 
&  & \ddots  &  &  &  \\ 
&  &  & \cos \theta _{n} & \sin \theta _{n} &  \\ 
&  &  & -\sin \theta _{n} & \cos \theta _{n} &  \\ 
&  &  &  &  & 1%
\end{array}%
\right)  \mid \theta _{i}\in \mathbb{R\}}
\end{equation*}%
with Lie algebra%
\begin{equation*}
\mathfrak{t}_{0}=\{\left( 
\begin{array}{cccccc}
0 & \theta _{1} &  &  &  &  \\ 
-\theta _{1} & 0 &  &  &  &  \\ 
&  & \ddots  &  &  &  \\ 
&  &  & 0 & \theta _{n} &  \\ 
&  &  & -\theta _{n} & 0 &  \\ 
&  &  &  &  & 0%
\end{array}%
\right)  \mid \theta _{i}\in \mathbb{R\}}.
\end{equation*}%
We write $\exp _{\operatorname{SO}(2n+1)}$ for the exponential map from $\mathfrak{t}_{0}$
onto $T_{0}$ and condense notation by writing $E_{k}$ for the element of $%
\mathfrak{t}$ given by%
\begin{equation*}
E_{k}=\operatorname{blockdiag}\bigg(\ 
\overset{k}{\overbrace{\left( 
\begin{array}{cc}
0 & 0 \\ 
0 & 0%
\end{array}%
\right) ,\ldots ,\left( 
\begin{array}{cc}
0 & 0 \\ 
0 & 0%
\end{array}%
\right) ,\left( 
\begin{array}{cc}
0 & 1 \\ 
-1 & 0%
\end{array}%
\right) }},\left( 
\begin{array}{cc}
0 & 0 \\ 
0 & 0%
\end{array}%
\right) ,\ldots ,\left( 
\begin{array}{cc}
0 & 0 \\ 
0 & 0%
\end{array}%
\right) ,0 \bigg) .
\end{equation*}

Writing $e_{k}$ for the $k^{\text{th}}$ standard basis vector in $\mathbb{R}%
^{n}$, observe that $\mathcal{A}\left( \cos \theta -\sin \theta
\,e_{2k-1}e_{k}\right) $ acts by the rotation $\left( 
\begin{array}{cc}
\cos 2\theta  & \sin 2\theta  \\ 
-\sin 2\theta  & \cos 2\theta 
\end{array}%
\right) $ in the $e_{2k-1}e_{k}$ plane. It follows that
\begin{equation*}
T=\{\left( \cos \theta _{1}-\sin \theta _{1}\,e_{1}e_{2}\right) \cdots
\left( \cos \theta _{n}-\sin \theta _{n}\,e_{2n-1}e_{2n}\right)  \mid \theta
_{k}\in \mathbb{R}\}
\end{equation*}
is a maximal torus of $\operatorname{Spin}(2n+1)$. If we identify 
$\mathfrak{t}$ with the Lie algebra of $T$ and write $\exp $ for the
exponential map of $\operatorname{Spin}(2n+1)$ taking $\mathfrak{t}$ onto 
$T$, then $\exp_{\SO(n)}=\mathcal{A}\circ\exp$.
%
%there is a commutative diagram
%$$
%\begin{diagram}
% & & T\\
% & \ruTo^{\exp} & \dTo_{ \mathcal{A}} \\
%\ft & \rTo_{\exp_{\SO(n)}} & T_{0}  
%\end{diagram}
%$$
It follows that%
\begin{equation*}
\exp \left( \theta E_{k}\right) =\left( \cos(\theta/2)-\sin(\theta/2)\,e_{2k-1}e_{2k}\right) .
\end{equation*}

Using the definitions, it is straightforward to check that%
\begin{eqnarray*}
v_{0} &=&0 \\
v_{1} &=&2\pi E_{1} \\
v_{k} &=&\pi \sum_{j=1}^{k}E_{j}
\end{eqnarray*}%
for $2\leq k\leq n$. Therefore $\exp v_{0}=1$, $\exp v_{1}=-1$, and $\exp
v_{k}=(-1)^{k}\prod_{j=1}^{k}e_{2j-1}e_{j}$. Of course, 
$\mathcal{O}_{0}=\{1\}$ and $\mathcal{O}_{1}=\{-1\}$ so 
$\operatorname{cat}(\mathcal{O}_{0})=\operatorname{cat}(\mathcal{O}_{1})=0$.

The other orbits are easy to describe, though calculating 
$\operatorname{cat}_{G}(\mathcal{O}_{k})$ is not easy.

\begin{proposition}
\label{thm:  orbits for spin odd}
For  $2\leq k\leq n$,
\begin{eqnarray*}
\mathcal{O}_{k} &\cong &\operatorname{Spin}(2k)/
\operatorname{Spin}(2k)\operatorname{Spin}(2n+1-2k)
\\
&\cong &\operatorname{SO}(2n+1)/\left(\operatorname{SO}(2k)\times \operatorname{SO}(2n+1-2k)\right)\cong \widetilde{{Gr}%
_{2k}}(\mathbb{R}^{2n+1})\text{,}
\end{eqnarray*}%
the Grassmannian of oriented $2k$-planes in $\mathbb{R}^{2n+1}$.
\end{proposition}

\begin{proof}
Since $\mathcal{A}(\exp v_{k})=\left( 
\begin{array}{cc}
-I_{2k} &  \\ 
& I_{n-2k}%
\end{array}%
\right) $, 
$$\mathcal{A}(\mathcal{O}_{k})\cong \operatorname{SO}(2n+1)/S\left(  \operatorname{O}(2k)\times
 \operatorname{O}(2n+1-2k)\right) \cong {Gr}_{2k}(\mathbb{R}^{2n+1}),$$ 
the Grassmannian of $2k$-planes in $\mathbb{R}^{2n+1}$. Moreover, $\mathcal{A}:%
\mathcal{O}_{k}\rightarrow \mathcal{A}(\mathcal{O}_{k})$ is a double cover.
To see this, observe that there is a Weyl group (isomorphic to $%
S_{n}\ltimes \mathbb{Z}_{2}^{n}$) element taking $v_{k}$ to $-\pi E_{1}+\pi
\sum_{j=2}^{k}E_{j}$ which exponentiates to $-\exp v_{k}$.

To prove the proposition, first observe that the stabilizer of $\exp v_{k}$
under conjugation must be contained in $S=\mathcal{A}^{-1}(S\left(
 \operatorname{O}(2k)\times  \operatorname{O}(2n+1-2k)\right) )=S\left( \operatorname{Pin}(2k)\operatorname{Pin}(2n+1-2k)\right) $. Since $\operatorname{Pin}(2k)\cap 
\operatorname{Pin}(2n+1-2k)\subseteq \mathbb{R}$, it follows that the
connected component of the identity of $S$ is $S_{0}=\operatorname{Spin}(2k)\operatorname{Spin}(n-2k)\cong 
\operatorname{Spin}(2k)\times \operatorname{Spin}(n-2k)/\{\pm (1,1)\}$ and the other component is
diffeomorphic to $\operatorname{Pin}(2k)_{1}\times 
\operatorname{Pin}(2n+1-2k)_{1}$ where $\operatorname{Pin}(j)_{1}$ is the non-identity
component of $\operatorname{Pin}(j)$. Recalling that the center of $%
\operatorname{Spin}(2k)$ is $\{\pm 1,\pm \exp v_{k}\}$, it follows that 
$S_{0}$ is contained in the stabilizer of $\exp v_{k}$. However, 
$\operatorname{Pin}(2k)_{1}$ anticommutes with $\exp v_{k}$ while 
$\operatorname{Pin}(2n+1-2k)_{1}$ commutes. Therefore, the stabilizer of $\exp
v_{k}$ is $S_{0}$. Finally, since $S_{0}=\mathcal{A}^{-1}(\operatorname{SO}(2k)\times
\operatorname{SO}(n-2k))$, the proof is complete.
\end{proof}

The relative cat calculation of $\mathcal{O}_{k}$ in 
$\operatorname{Spin}(2n+1)$ is not known.

\subsection{$G=\operatorname{Spin}(2n)$}

A maximal torus\ $T_{0}$ for $\operatorname{SO}(2n)$ is given by%
\begin{equation*}
T_{0}=\{\left( 
\begin{array}{ccccc}
\cos \theta _{1} & \sin \theta _{1} &  &  &  \\ 
-\sin \theta _{1} & \cos \theta _{1} &  &  &  \\ 
&  & \ddots  &  &  \\ 
&  &  & \cos \theta _{n} & \sin \theta _{n} \\ 
&  &  & -\sin \theta _{n} & \cos \theta _{n}%
\end{array}%
\right)  \mid \theta _{i}\in \mathbb{R\}}
\end{equation*}%
with Lie algebra%
\begin{equation*}
\mathfrak{t}=\{\left( 
\begin{array}{ccccc}
0 & \theta _{1} &  &  &  \\ 
-\theta _{1} & 0 &  &  &  \\ 
&  & \ddots  &  &  \\ 
&  &  & 0 & \theta _{n} \\ 
&  &  & -\theta _{n} & 0%
\end{array}%
\right)  \mid \theta _{i}\in \mathbb{R\}}.
\end{equation*}

As before, write%
\begin{equation*}
E_{k}=\operatorname{blockdiag}\left( \overset{k}{\overbrace{\left( 
\begin{array}{cc}
0 & 0 \\ 
0 & 0%
\end{array}%
\right) ,\ldots ,\left( 
\begin{array}{cc}
0 & 0 \\ 
0 & 0%
\end{array}%
\right) ,\left( 
\begin{array}{cc}
0 & 1 \\ 
-1 & 0%
\end{array}%
\right) }},\left( 
\begin{array}{cc}
0 & 0 \\ 
0 & 0%
\end{array}%
\right) ,\ldots ,\left( 
\begin{array}{cc}
0 & 0 \\ 
0 & 0%
\end{array}%
\right) \right) .
\end{equation*}

From the definitions, it is straightforward to check that%
\begin{eqnarray*}
v_{0} &=&0 \\
v_{1} &=&2\pi E_{1} \\
v_{k} &=&\pi \sum_{j=1}^{k}E_{j} \\
v_{n-1} &=&\pi \sum_{j=1}^{n-1}E_{j}-\pi E_{n}
\end{eqnarray*}%
for $2\leq k\leq n$, $k\neq n-1$. Therefore $\exp v_{0}=1$, $\exp v_{1}=-1$%
, $\exp v_{k}=(-1)^{k}\prod_{j=1}^{k}e_{2j-1}e_{j}$, and $\exp
v_{n-1}=(-1)^{n-1}\prod_{j=1}^{n}e_{2j-1}e_{j}$.Of course, $%
\mathcal{O}_{0}=\{1\}$ and $\mathcal{O}_{1}=\{-1\}$ so $\operatorname{cat}%
(\mathcal{O}_{0})=\operatorname{cat}(\mathcal{O}_{1})=0$. As in Propostion~\ref{thm: 
orbits for spin odd}, the remaining conjugacy classes are 
\begin{eqnarray*}
\mathcal{O}_{k} &\cong &\operatorname{Spin}(2k)/\operatorname{%
Spin}_{2k}(\mathbb{R})\operatorname{Spin}(2n-2k) \\
&\cong &\operatorname{SO}(2n)/\operatorname{SO}(2k)\times \operatorname{SO}(2n-2k)\cong \widetilde{{Gr}%
_{2k}}(\mathbb{R}^{2n})\text{,}
\end{eqnarray*}%
the Grassmannian of oriented $2k$-planes in $\mathbb{R}^{2n}$. Again, the
relative category in $\operatorname{Spin}(2n)$ is not known.

\let\section=\oldsection
%\bibliographystyle{amsmath}
%\bibliography{jabbrev,refs}

\def\germ{\mathfrak}\def\cprime{$'$}\def\scr{\mathcal}

\end{document}